\documentclass[10pt]{article}
\usepackage{titlesec}
\titleformat{\subsection}{\normalsize\itshape}{\thesubsection}{1em}{}
\usepackage{amssymb, amsmath, amsthm, amsfonts, enumerate,bm,nicefrac}
\usepackage{mathtools}
\usepackage{amsmath,setspace,scalefnt}
\usepackage[dvipsnames,svgnames,table]{xcolor}
\usepackage{graphicx,tikz,caption,subcaption}
\usetikzlibrary{patterns}
\usetikzlibrary{shapes}
\usepackage{fullpage}
\usepackage[hidelinks]{hyperref}
\usepackage{authblk}
\usepackage{enumitem,verbatim}
\usepackage{multicol}
\usepackage{float}
\usepackage{cleveref}
\usepackage{bm}

\usepackage[margin=2.5cm]{geometry}

\definecolor{gray}{rgb}{0.25, 0.25, 0.25}
\usepackage{tikz}
\usepackage{booktabs}
\newtheorem{theorem}{Theorem}[section]
\newtheorem{lemma}[theorem]{Lemma}

\newtheorem{corollary}[theorem]{Corollary}

\newtheorem*{observation*}{Observation}

\newtheorem{problem}[theorem]{Problem}
\newtheorem*{problem*}{Problem}

\newtheorem*{question*}{Question}

\theoremstyle{definition}

\newtheorem{fact}[theorem]{Fact}
\newenvironment{definition*}
  {
   \innerdefinition}
  {\endinnerdefinition}

\theoremstyle{remark}

\DeclareMathOperator{\supp}{supp}

\title{\Large\bf The spectral inducibility of graphs}

\author[1]{Liying Kang\thanks{Email: \texttt{lykang@shu.edu.cn}}}
\author[2]{Xizhi Liu\thanks{Email: \texttt{liuxizhi@ustc.edu.cn}}}
\author[1]{Yongchun Lu\thanks{ Corresponding author. Email: \texttt{luyongchun@shu.edu.cn}}}

\affil[1]{\small Department of Mathematics, Shanghai University, Shanghai, China}
\affil[2]{\small School of Mathematical Sciences, University of Science and Technology of China, Hefei, Anhui, China}
\date{\today}

\begin{document}
\maketitle
\begin{abstract}

We introduce a spectral version of the classical inducibility problem. Given an
$\ell$-vertex graph $F$ and an $n$-vertex graph $G$, let $H_F(G)$ be the
$\ell$-uniform hypergraph whose edges are the $\ell$-sets inducing a copy of $F$ in
$G$. We study the maximum possible $\alpha$-spectral radius of $H_F(G)$ over
all $n$-vertex graphs $G$. For fixed $G$, this spectral parameter tends to
$\ell!$ times the number of induced copies of $F$ in $G$ as $\alpha\to\infty$,
and therefore refines the usual induced-copy count.

Our main result is a spectral analogue of the Brown--Sidorenko reduction: for
every complete multipartite graph $F$, every $n$, and every $\alpha\ge1$, a
spectral extremal graph can be chosen to be complete multipartite. We also show that
the leading asymptotic constant is the ordinary inducibility $i(F)$, and
obtain exact multipartite reductions for stars $K_{1,t}$ and balanced
complete $r$-partite graphs $K_{a,\ldots,a}$ with $r\le 2^a-1$.
\end{abstract}

\section{Introduction}

A fundamental problem in extremal graph theory, introduced by Pippenger and
Golumbic~\cite{pippenger1975inducibility}, is to determine the maximum number
of induced copies of a fixed graph $F$ among all $n$-vertex graphs. Given two
graphs $F$ and $G$, let $I(F,G)$ denote the number of subsets $X\subseteq V(G)$
such that $X$ induces a copy of $F$ in $G$. Write $v(G)$ for the number of
vertices of $G$. For every positive integer $n$, define
\begin{align*}
I(F,n)&\coloneqq\max\{I(F,G): v(G)=n\}.
\end{align*}
The \emph{inducibility} of $F$ is defined by
\begin{align*}
i(F)&\coloneqq\lim_{n\to \infty}{I(F,n)}/{\tbinom{n}{v(F)}}.
\end{align*}
A simple averaging argument shows that the ratio is non-increasing in $n$, and
hence the limit exists. Although the definition is elementary, determining
$I(F,n)$ or even $i(F)$ is notoriously difficult. Exact or asymptotic results
are known only for special families of graphs. Related small-graph inducibility
problems have been investigated
in~\cite{balogh2016maximum,exoo1986dense,hirst2014inducibility,pikhurko2019strong},
and some of these works use Razborov's flag algebra method~\cite{razborov2007flag}.
Notably, the inducibility of $P_4$, the path on four vertices, is still open.
The best known lower and upper bounds were obtained in
\cite{evenzohar2015note}.

Complete multipartite graphs form one of the most natural and important
classes in the study of inducibility. For positive integers $k$ and $n$, let
$T_k(n)$ denote the \emph{Tur\'an graph},
namely the complete $k$-partite graph on $n$ vertices whose part sizes differ
by at most one. For $rt$-vertex Tur\'an graphs $T_r(rt)$, Bollob\'{a}s, Egawa,
Harris and Jin~\cite{bollobas1995maximal} proved that, when $t>1+\log r$,
the maximum number of induced copies is attained by the Tur\'{a}n graph
$T_r(n)$. This direction was further developed by Hatami, Hirst and
Norine~\cite{hatami2014inducibility}, who showed more generally that the
inducibility extremizers of sufficiently large balanced blow-ups are
asymptotically blow-ups of the original graph. In a fundamental paper, Brown
and Sidorenko~\cite{brown1994inducibility} proved that, when $F$ is
complete bipartite, the extremizer can be chosen to be complete
bipartite. Moreover, they showed that for every complete multipartite graph
$F$ and every $n\geq 1$, there exists an $n$-vertex complete multipartite graph
$G$ satisfying $I(F,G)=I(F,n)$. However, determining the optimal number of
parts of $G$ and the ratios of its part sizes remains a difficult problem.
Thus the inducibility of complete multipartite graphs is still far from being
completely understood. For general graphs, results of Yuster~\cite{yuster2019exact} and
Fox, Huang and Lee~\cite{fox2017solution} showed that, for almost all graphs
$F$, the inducibility of $F$ is given by
$i(F)=v(F)!/(v(F)^{v(F)}-v(F))$.
Recent progress in this direction was obtained by
Yuster~\cite{yuster2026inducibility}, who studied the inducibility of
Tur\'{a}n graphs. Liu, Ma and Zhu~\cite{liu2026inducibility} further studied
almost balanced complete multipartite graphs and showed that, for sufficiently
large $n$, the extremal graph is the Tur\'{a}n graph. In a related direction,
Liu, Mubayi and Reiher~\cite{liu2023feasible} studied feasible regions of
induced graphs. This framework refines the
ordinary inducibility problem by considering, for each possible edge density,
the maximum induced density of a fixed graph.

It is well known that, in extremal combinatorics, spectral counterparts of
counting parameters are often closely related to the corresponding counting
problems (see, for example,
\cite{friedman1995small,friedman1995second,keevash2014spectral,liu2025spectral,zheng2025spectral}).
This viewpoint has been particularly fruitful in spectral extremal graph theory:
Nikiforov's work gives spectral versions of classical Tur\'an-type and
Erd\H{o}s--Stone--Bollob\'as phenomena~\cite{nikiforov2002some,nikiforov2009spectral}, while the $p$-spectral radius studied by Nikiforov and coauthors interpolates between
extremal edge counts and spectral radii and yields asymptotic equivalences for
hereditary properties~\cite{kang2014extremal,nikiforov2014some}. Motivated by
this perspective, we develop a spectral analogue of the inducibility problem.
Let $F$ be an $\ell$-vertex graph. For an $n$-vertex graph $G$, define the
$\ell$-uniform hypergraph $H_F(G)$ on vertex set $V(G)$ by
\begin{align*}
H_F(G)&\coloneqq\{S\subseteq V(G): G[S]\cong F\}.
\end{align*}
Thus the classical inducibility problem measures $H_F(G)$ simply by its
number of edges. For $\mathbf{x}=(x_1,\ldots,x_n)\in\mathbb{R}^n$, set
\begin{align*}
P_{F,G}(\mathbf{x})&\coloneqq \ell!\sum_{e\in H_F(G)}\prod_{i\in e}x_i.
\end{align*}
The hypergraph eigenvalues used here are based on the definitions of Friedman
and Wigderson~\cite{friedman1995second}. For $\alpha\ge1$, define
\begin{align*}
\lambda_{\alpha}^{\mathrm{ind}}(F,G)&\coloneqq\max\left\{P_{F,G}(\mathbf{x}):\mathbf{x}\in\mathbb{S}_\alpha^{n-1}\right\}, 
\end{align*}
where $\mathbb{S}_\alpha^{n-1} \coloneqq \left\{\mathbf{x}\in\mathbb{R}^n: |x_1|^\alpha+\cdots+|x_n|^\alpha=1\right\}$. 
Since replacing a vector by its coordinatewise absolute value cannot decrease
$P_{F,G}$, the maximum always has a nonnegative optimizer. Let
\begin{align*}
\Delta_\alpha^{n-1}
&\coloneqq\left\{\mathbf{x}\in\mathbb{R}_{\ge0}^n:
|x_1|^\alpha+\cdots+|x_n|^\alpha=1\right\}.
\end{align*}
We define
\begin{align*}
\operatorname{OPT}_{\alpha,F}(G)
&\coloneqq\left\{\mathbf{x}\in\Delta_\alpha^{n-1}:
P_{F,G}(\mathbf{x})=\lambda_\alpha^{\mathrm{ind}}(F,G)\right\}.
\end{align*}

The \emph{$\alpha$-spectral inducibility number} of $F$ is
\begin{align*}
\operatorname{spec}_\alpha^{\operatorname{ind}}(F,n)&\coloneqq\max\{\lambda_\alpha^{\mathrm{ind}}(F,G):v(G)=n\}.
\end{align*}

With this notation, it is natural to pose the following problem.
\begin{problem}
Given a graph $F$, determine the value of
$\operatorname{spec}_{\alpha}^{\mathrm{ind}}(F,n)$.
\end{problem}

This problem contains the classical inducibility problem as a limiting case
and, when $\alpha=1$, is closely related to the classical optimization
problem arising from blow-ups; see Section~\ref{sec:prelim} for details.

Among the early results on inducibility, a substantial part concerns
complete multipartite graphs. Following this line of work, we study the
corresponding extremal problem for the spectral inducibility of complete
multipartite graphs. Motivated by these classical works, we obtain the
following spectral analogues.
\medskip
\begin{theorem}\label{thm:cmp-ext}
Let $F$ be a complete multipartite graph with $\ell$ vertices. The following hold.
\begin{enumerate}[label=(\roman*)]
    \item \label{thmi} For every positive integer $n$ and every $\alpha\ge1$, there exists an $n$-vertex complete multipartite graph $G$ such that $\lambda_\alpha^{\mathrm{ind}}(F,G)=\operatorname{spec}_\alpha^{\operatorname{ind}}(F,n)$.
    \item \label{thmii} For every positive integer $n$ and every $\alpha\ge1$, if an $n$-vertex complete multipartite graph $G$ satisfies $\lambda_\alpha^{\mathrm{ind}}(F,G)=\operatorname{spec}_\alpha^{\operatorname{ind}}(F,n)$, then there exists a vector $\mathbf{x}\in \operatorname{OPT}_{\alpha,F}(G)$ whose entries are identical on each part of $G$.
    \item \label{thmiii} For every fixed $\alpha\ge1$, as $n\to\infty$, we have $\operatorname{spec}_\alpha^{\mathrm{ind}}(F,n)=\left(i(F)+o(1)\right)n^{\ell(1-1/\alpha)}$.
\end{enumerate}
\end{theorem}
\medskip

We also prove exact reductions to complete multipartite graphs for two special
families.
\medskip
\begin{theorem}\label{thm:star}
For every integer $t\ge 2$, every $\alpha\ge1$ and every positive integer $n$, we have
\begin{align*}
\operatorname{spec}_\alpha^{\operatorname{ind}}(K_{1,t},n)&=
\max_{\substack{n_1,n_2\ge0\\ n_1+n_2=n}}
\left\{\lambda_\alpha^{\mathrm{ind}}(K_{1,t},K_{n_1,n_2})\right\},
\end{align*}
where parts of size zero are omitted.
\end{theorem}
\medskip

\begin{theorem}\label{thm:multipartite-balanced}
Let $a\ge 2$ and $r\ge 2$ be integers with $r\le 2^a-1$. Let $F$ be the complete
$r$-partite graph whose parts all have size $a$. For every $\alpha\ge1$ and
every integer $n\ge 1$, we have
\begin{align*}
\operatorname{spec}_\alpha^{\operatorname{ind}}(F,n)&=
\max_{\substack{n_1,\ldots,n_r\ge0\\ n_1+\cdots+n_r=n}}
\left\{\lambda_\alpha^{\mathrm{ind}}(F,K_{n_1,n_2,\ldots,n_r})\right\},
\end{align*}
where parts of size zero are omitted.
\end{theorem}

When $r=2$, Theorem~\ref{thm:multipartite-balanced} applies to the balanced
complete bipartite graph $K_{a,a}$ for every $a\geq 2$, since
$2\leq 2^a-1$ for all $a\geq 2$. Hence, in this case, there always exists a
complete bipartite graph attaining
$\operatorname{spec}_\alpha^{\mathrm{ind}}(K_{a,a},n)$ for every $a\ge2$.

\medskip
The rest of the paper is organized as follows. In Section~\ref{sec:prelim}, we
collect basic observations and notation. In Section~\ref{sec:cmp-ext-proof}, we
prove Theorem~\ref{thm:cmp-ext}. Finally, Sections~\ref{sec:star-proof}
and~\ref{sec:multipartite-proof} are devoted to the proofs of
Theorems~\ref{thm:star} and~\ref{thm:multipartite-balanced}, respectively.

\section{Preliminaries}\label{sec:prelim}

Let $G$ be a graph.
Throughout the paper, when some part sizes in $K_{n_1,\ldots,n_q}$ are zero,
those parts are omitted. Unless explicitly stated otherwise, in formulas
involving division by $n_i$, all part sizes are assumed to be positive.
For a vector $\mathbf{x}=(x_i)_{i\in V(G)}$, write
\begin{align*}
\supp(\mathbf{x})&\coloneqq\{v\in V(G):x_v>0\}.
\end{align*}
Two nonadjacent vertices are called \emph{twins} if they have the same
neighborhood. We use $\mathcal{T}(G,\mathbf{x})$ to denote the number of
twin-pairs in $G[\supp(\mathbf{x})]$. An \emph{equivalence class} is a maximal
subset of vertices such that any two vertices in the class are twins.

We next record two basic observations about the spectral parameter. This
spectral problem refines the classical inducibility problem. Indeed, for a
fixed $\ell$-vertex graph $F$ and an $n$-vertex graph $G$, substituting the
vector $(n^{-1/\alpha},\dots,n^{-1/\alpha})\in\Delta_\alpha^{n-1}$ into
$P_{F,G}$ gives
$\lambda_\alpha^{\mathrm{ind}}(F,G)\ge \ell!|H_F(G)|n^{-\ell/\alpha}$, while $P_{F,G}(\mathbf{x})\le \ell!|H_F(G)|$ for every feasible vector $\mathbf{x}$. Hence
\begin{align*}
\lim_{\alpha\to\infty}\lambda_\alpha^{\mathrm{ind}}(F,G)&=\ell!I(F,G).
\end{align*}
When $\alpha=1$, determining $\lambda_1^{\mathrm{ind}}(F,G)$ is exactly the
classical optimization problem arising from blow-ups. If
$\mathbf{x}=(x_v)_{v\in V(G)}$ is feasible and $N$ is sufficiently large, then
replacing each vertex $v\in V(G)$ with an independent class of size
$(x_v+o(1))N$ gives a blow-up hypergraph with
$\left(P_{F,G}(\mathbf{x})/\ell!+o(1)\right)N^\ell$ edges.

\begin{fact}\label{lem:uniform-lower-bound}
Let $F$ be an $\ell$-vertex graph and let $G_n$ be an $n$-vertex graph.
Then
\begin{align*}
    \lambda_\alpha^{\mathrm{ind}}(F,G_n)
    \ge P_{F,G_n}(n^{-1/\alpha}, \ldots, n^{-1/\alpha})
    = \frac{\ell!}{n^{\ell/\alpha}} \cdot I(F,G_n).
\end{align*}
In particular, as $n\to\infty$,
\begin{align*}
\operatorname{spec}_\alpha^{\mathrm{ind}}(F,n)
&\ge i(F)n^{\ell(1-1/\alpha)}
+O\!\left(n^{\ell(1-1/\alpha)-1}\right).
\end{align*}
\end{fact}

We also use the following notation in the proofs. For two nonadjacent vertices
$u,v\in V(G)$, let $G_{u\to v}$ denote the graph obtained from $G$ by the Zykov
symmetrization of $u$ toward $v$, that is, by replacing the neighborhood of
$u$ with the neighborhood of $v$. For a vector
$\mathbf{t}=(t_1,\ldots,t_n)$ and $\alpha\ge1$, let
$\mathbf{t}_{u,v}=(t_1',\ldots,t_n')$ be defined by
\begin{align*}
t'_u=t'_v&\coloneqq\left(\frac{t_u^\alpha+t_v^\alpha}{2}\right)^{1/\alpha},
\qquad\text{and}\qquad
t'_w\coloneqq t_w \quad \text{for } w\ne u,v.
\end{align*}
Then $\|\mathbf{t}\|_\alpha=\|\mathbf{t}_{u,v}\|_\alpha$.

For a positive integer $m$, write $[m]\coloneqq\{1,\ldots,m\}$. 
For integers $a\le b$, write $[a,b]\coloneqq\{a,a+1,\ldots,b\}$. 
For integers $p,m\ge1$, define
\begin{align*}
[m]_p&\coloneqq\{(i_1,\ldots,i_p)\in[m]^p: i_1,\ldots,i_p \text{ are pairwise distinct}\}.
\end{align*}
For positive integers $d_1,\ldots,d_s$, recall that
$\operatorname{sym}(d_1,\ldots,d_s)$ denotes the product of the factorials of
the multiplicities of the distinct values among $d_1,\ldots,d_s$.
For a complete $r$-partite graph $F=K_{a_1,\ldots,a_r}$ with
$\ell=a_1+\cdots+a_r$, define
\begin{align*}
\kappa_F&\coloneqq
\frac{\ell!}{a_1!\cdots a_r!\operatorname{sym}(a_1,\ldots,a_r)}.
\end{align*}

\begin{lemma}\label{lem:const}
Let $F$ be a complete multipartite graph, and let $K$ be a complete
multipartite graph. There exists a nonnegative optimal vector $\mathbf{x}$
for $\lambda_\alpha^{\mathrm{ind}}(F,K)$ such that the entries of
$\mathbf{x}$ are identical on each part of $K$.
\end{lemma}

\begin{proof}
Among all nonnegative optimal vectors for $\lambda_\alpha^{\mathrm{ind}}(F,K)$,
choose one, say $\mathbf{x}$, which minimizes
\begin{align*}
\Phi(\mathbf{x})&\coloneqq\sum_{v\in V(K)}x_v^{2\alpha}.
\end{align*}
We show that this vector $\mathbf{x}$ satisfies the statement.
Suppose otherwise that there are two vertices $u,v$ in the same part with
$x_u\ne x_v$. Set
\begin{align*}
a&\coloneqq\left(\frac{x_u^\alpha+x_v^\alpha}{2}\right)^{1/\alpha}.
\end{align*}
For the feasible vector $\mathbf{x}_{u,v}$, we show that
$P_{F,K}(\mathbf{x}_{u,v})\ge P_{F,K}(\mathbf{x})$ but
$\Phi(\mathbf{x}_{u,v})< \Phi(\mathbf{x})$, which gives a contradiction. On
the one hand, since $u$ and $v$ are in the same part of
$K$, we have
$P_{F,K}(\mathbf{x})=A+B(x_u+x_v)+Cx_ux_v$ and
$P_{F,K}(\mathbf{x}_{u,v})=A+2aB+a^2C$,
where $B, C\ge 0$. Using the power-mean inequality and the AM--GM
inequality, we obtain
\[
2a\ge x_u+x_v,
\qquad\text{and}\qquad
a^2\ge x_ux_v.
\]
Hence $P_{F,K}(\mathbf{x}_{u,v})\ge P_{F,K}(\mathbf{x})$.
Since $\mathbf{x}_{u,v}$ is feasible and $\mathbf{x}$ is optimal, we also have
$P_{F,K}(\mathbf{x}_{u,v})\le \lambda_\alpha^{\mathrm{ind}}(F,K)
=P_{F,K}(\mathbf{x})$. Thus equality holds, and $\mathbf{x}_{u,v}$ is also
optimal.

On the other hand, since $t^2$ is strictly convex, by Jensen's inequality,
we obtain
\begin{align*}
2a^{2\alpha}
&=2\left(\frac{x_u^\alpha+x_v^\alpha}{2}\right)^2<x_u^{2\alpha}+x_v^{2\alpha},
\end{align*}
where the inequality holds as $x_u\ne x_v$. Thus
$\Phi(\mathbf{x}_{u,v})<\Phi(\mathbf{x})$, a contradiction. This completes
the proof.
\end{proof}

\begin{corollary}\label{cor:exact-complete-multipartite}
Let $F=K_{a_1,\ldots,a_r}$ be a complete $r$-partite graph with
$\ell=a_1+\cdots+a_r$. For every complete $q$-partite graph
$K_{n_1,\ldots,n_q}$ with $n_i>0$ for all $i\in[q]$, we have
\begin{align*}
\lambda_\alpha^{\mathrm{ind}}(F,K_{n_1,\ldots,n_q})
&=\frac{\ell!}{\operatorname{sym}(a_1,\ldots,a_r)}
\max_{\substack{(\mu_1,\ldots,\mu_q)\in \Delta_1^{q-1}}}\left\{
\sum_{(i_1,\ldots,i_r)\in [q]_r}
\prod_{j\in[r]}
\tbinom{n_{i_j}}{a_j}
\left(\tfrac{\mu_{i_j}}{n_{i_j}}\right)^{a_j/\alpha}
\right\}.
\end{align*}
\end{corollary}

\begin{proof}
By Lemma~\ref{lem:const}, it suffices to consider feasible vectors that are
constant on each part of $K_{n_1,\ldots,n_q}$. For such a vector
$\mathbf{x}$, let $w_i$ denote the common value of $x_v$ on all vertices
in the $i$-th part, and set $\mu_i\coloneqq n_iw_i^\alpha$. Then
$\sum_{i\in[q]}\mu_i=1$, and
\begin{align*}
P_{F,K_{n_1,\ldots,n_q}}(\mathbf{x})
&=\frac{\ell!}{\operatorname{sym}(a_1,\ldots,a_r)}
\sum_{(i_1,\ldots,i_r)\in[q]_r}
\prod_{j\in[r]}\tbinom{n_{i_j}}{a_j}w_{i_j}^{a_j}\\
&=\frac{\ell!}{\operatorname{sym}(a_1,\ldots,a_r)}
\sum_{(i_1,\ldots,i_r)\in[q]_r}
\prod_{j\in[r]}
\tbinom{n_{i_j}}{a_j}
\left(\tfrac{\mu_{i_j}}{n_{i_j}}\right)^{a_j/\alpha}.
\end{align*}
Conversely, every $\boldsymbol{\mu}\in\Delta_1^{q-1}$ gives a feasible
vector by setting $w_i=(\mu_i/n_i)^{1/\alpha}$ on the $i$-th part. Taking
the maximum over all such $\boldsymbol{\mu}$ gives the formula.
\end{proof}

\begin{lemma}\label{lem:upper-complete-multipartite}
Let $F=K_{a_1,\ldots,a_r}$ be a complete $r$-partite graph with
$\ell=a_1+\cdots+a_r$, and let $G=K_{n_1,\ldots,n_q}$ be an $n$-vertex
complete $q$-partite graph with positive part sizes. Let
$\boldsymbol{\mu}=(\mu_1,\ldots,\mu_q)$ be
an optimal vector in Corollary~\ref{cor:exact-complete-multipartite}, and
define $\mathbf{z}=(z_1,\ldots,z_q)$ by 
\begin{align*}
z_i&\coloneqq\left(\frac{n_i}{n}\right)^{1-1/\alpha}\mu_i^{1/\alpha}
\quad \text{for}\quad i\in[q].
\end{align*}
Then $\sum_{i\in[q]}z_i\le1$ and
\begin{align*}
\lambda_\alpha^{\mathrm{ind}}(F,G)
\le
\kappa_F n^{\ell(1-1/\alpha)}
\sum_{(i_1,\ldots,i_r)\in[q]_r}
\prod_{j\in[r]}z_{i_j}^{a_j}
\le i(F)n^{\ell(1-1/\alpha)}.
\end{align*}
\end{lemma}

\begin{proof}
By the definition of $\mathbf{z}$ when $\alpha=1$, and by H\"older's
inequality when $\alpha>1$, we have
\begin{align*}
\sum_{i\in[q]} z_i
=\sum_{i\in[q]}\left(\frac{n_i}{n}\right)^{1-1/\alpha}\mu_i^{1/\alpha}
\le\Big(\sum_{i\in[q]}\frac{n_i}{n}\Big)^{1-1/\alpha}
\Big(\sum_{i\in[q]}\mu_i\Big)^{1/\alpha}
=1.
\end{align*}
For the upper bound, since $\binom{m}{p}\le m^p/p!$ for every $m,p\ge1$,
we have
\begin{align*}
\binom{n_i}{p}\left(\frac{\mu_i}{n_i}\right)^{p/\alpha}
&\le \frac{n^{p-p/\alpha}}{p!}z_i^p.
\end{align*}
Using Corollary~\ref{cor:exact-complete-multipartite}, we obtain
\begin{align*}
\lambda_\alpha^{\mathrm{ind}}(F,G)
&\le
\frac{\ell!}{\operatorname{sym}(a_1,\ldots,a_r)}
\sum_{(i_1,\ldots,i_r)\in [q]_r}
\prod_{j\in[r]}
\frac{n^{a_j(1-1/\alpha)}}{a_j!}z_{i_j}^{a_j}
=
\kappa_F n^{\ell(1-1/\alpha)}
\sum_{(i_1,\ldots,i_r)\in[q]_r}
\prod_{j\in[r]}z_{i_j}^{a_j}.
\end{align*}
It remains to justify the last inequality. Let
$z_{q+1}=1-\sum_{i\in[q]}z_i$. Choose nonnegative integers $m_1,\ldots,m_{q+1}$
with $\sum_{i\in[q+1]}m_i=N$ and $m_i/N\to z_i$, and let
$H_N=K_{m_1,\ldots,m_{q+1}}$, omitting parts of size zero. Then
\begin{align*}
\frac{I(F,H_N)}{\binom{N}{\ell}}
&\to
\kappa_F
\sum_{(i_1,\ldots,i_r)\in[q+1]_r}
\prod_{j\in[r]}z_{i_j}^{a_j}
\end{align*}
as $N\to\infty$. Since $I(F,H_N)\le I(F,N)$, the limit above is at most
$i(F)$. Therefore
\begin{align*}
\kappa_F
\sum_{(i_1,\ldots,i_r)\in[q]_r}
\prod_{j\in[r]}z_{i_j}^{a_j}
&\le
\kappa_F
\sum_{(i_1,\ldots,i_r)\in[q+1]_r}
\prod_{j\in[r]}z_{i_j}^{a_j}
\le i(F),
\end{align*}
which proves the desired bound.
This completes the proof.
\end{proof}

\section{Proof of Theorem~\ref{thm:cmp-ext}}\label{sec:cmp-ext-proof}
We prove Theorem~\ref{thm:cmp-ext} in this section. Let $F$ be a complete
multipartite graph. Let $G$ be an $n$-vertex graph such that
$\lambda_\alpha^{\mathrm{ind}}(F,G)
=\operatorname{spec}_\alpha^{\mathrm{ind}}(F,n)$.
Let $\mathbf{x}=(x_v)_{v\in V(G)}\in \operatorname{OPT}_{\alpha,F}(G)$ be a
nonnegative vector. For $u,v\in V(G)$, write
\begin{align*}
P_{F,G}(\mathbf{x})&=A+x_uB_u+x_vB_v+x_ux_vC,
\end{align*}
where $A$, $B_u$, $B_v$ and $C$ are independent of $x_u$ and $x_v$.

\medskip
The key ingredient of the proof is the following lemma.
\medskip
\begin{lemma}\label{lem:sym}
Let $u,v\in V(G)$ be two nonadjacent vertices with $x_ux_v\neq0$. If $B_v\ge B_u$, then 
\begin{align*}
P_{F,G_{u\to v}}(\mathbf{x}_{u,v})-P_{F,G}(\mathbf{x})\ge 0.
\end{align*}
Moreover, the inequality is strict if $B_v>B_u$.
\end{lemma}

\begin{proof}
For convenience, set
\begin{align*}
a&\coloneqq\left(\frac{x_u^\alpha+x_v^\alpha}{2}\right)^{1/\alpha}.
\end{align*}
Since $u$ and $v$ have the same neighbors in $G_{u\to v}$, we have
\begin{align*}
P_{F,G_{u\to v}}(\mathbf{x}_{u,v})=A+2aB_v+a^2C',
\end{align*}
where $C'$ is independent of $a$. Here the term $A$ is unchanged because
Zykov symmetrization only changes edges incident with $u$, and therefore it
does not affect any $\ell$-set avoiding both $u$ and $v$. The coefficient of
the $v$-only terms remains $B_v$. The coefficient of the $u$-only terms is
also $B_v$: for every $Y\subseteq V(G)\setminus\{u,v\}$ with $|Y|=\ell-1$,
\[
G_{u\to v}[\{u\}\cup Y]\cong F
\quad\Longleftrightarrow\quad
G[\{v\}\cup Y]\cong F.
\]

We claim that $C'\ge C$. Indeed, let $X$ be the vertex set of an induced copy
of $F$ in $G$ containing both $u$ and $v$. Since $u$ and $v$ are nonadjacent
and $F$ is complete multipartite, they must belong to the same part of this
copy. Hence, for every $w\in X\setminus\{u,v\}$, the adjacencies of $u$ and
$v$ to $w$ coincide inside $G[X]$. After symmetrizing $u$ toward $v$, the
adjacency between $u$ and every vertex of $X\setminus\{u,v\}$ is unchanged,
and all other adjacencies inside $X$ are unchanged as well. Thus $X$ still
induces a copy of $F$. Therefore every monomial contributing to $C$ also
contributes to $C'$, so $C'\ge C$.

By the power-mean inequality and the AM--GM inequality, we have
\begin{align*}
2a\ge x_u+x_v,
\text{ and } a^2\ge x_ux_v.
\end{align*}
Since $C$ and $C'$ are nonnegative coefficients and $C'\ge C$, we also have
$a^2C'-x_ux_vC\ge0$.
Hence
\begin{align*}
    P_{F,G_{u\to v}}(\mathbf{x}_{u,v})-P_{F,G}(\mathbf{x})&= 2aB_v+a^2C'-x_uB_u-x_vB_v-x_ux_vC\\[2mm]
    &\ge a^2C'-x_ux_vC+(2a-x_u-x_v)B_v\ge0,
\end{align*}
where the first inequality follows from $B_v\ge B_u$.
Moreover, since $x_u>0$, the first inequality is strict if $B_v>B_u$. The result follows.
\end{proof}
\medskip
Recall that $G$ is an $n$-vertex graph satisfying $\lambda_\alpha^{\mathrm{ind}}(F,G)=\operatorname{spec}_\alpha^{\mathrm{ind}}(F,n)$.
\medskip
\begin{corollary}\label{B_vB_u}
Let $u,v\in V(G)$ be two nonadjacent vertices with $x_ux_v\neq0$. Then $B_v=B_u$.
\end{corollary}
\begin{proof}
If $B_v>B_u$, then Lemma~\ref{lem:sym} gives
\[
P_{F,G_{u\to v}}(\mathbf{x}_{u,v})
>P_{F,G}(\mathbf{x})
=\operatorname{spec}_\alpha^{\mathrm{ind}}(F,n),
\]
contradicting the definition of
$\operatorname{spec}_\alpha^{\mathrm{ind}}(F,n)$. Similarly, if
$B_u>B_v$, applying Lemma~\ref{lem:sym} with the roles of $u$ and $v$
interchanged gives the same contradiction for $G_{v\to u}$. Hence
$B_u=B_v$.
\end{proof}
\medskip
Combining Lemma~\ref{lem:sym} and Corollary~\ref{B_vB_u}, we obtain the following result.
\medskip
\begin{corollary}\label{core}
Let $u,v\in V(G)$ be two nonadjacent vertices with $x_ux_v\neq0$. Then
\begin{align*}
\lambda_\alpha^{\mathrm{ind}}(F,G_{u\to v})=\lambda_\alpha^{\mathrm{ind}}(F,G_{v\to u})=\operatorname{spec}_\alpha^{\mathrm{ind}}(F,n),
\end{align*}
and $\mathbf{x}_{u,v}$ belongs to both $\operatorname{OPT}_{\alpha,F}(G_{u\to v})$ and $\operatorname{OPT}_{\alpha,F}(G_{v\to u})$.
\end{corollary}
\begin{proof}
By Corollary~\ref{B_vB_u} and Lemma~\ref{lem:sym},
\[
P_{F,G_{u\to v}}(\mathbf{x}_{u,v})\ge P_{F,G}(\mathbf{x})
=\operatorname{spec}_\alpha^{\mathrm{ind}}(F,n).
\]
Since $\lambda_\alpha^{\mathrm{ind}}(F,G_{u\to v})\le
\operatorname{spec}_\alpha^{\mathrm{ind}}(F,n)$ and
$\lambda_\alpha^{\mathrm{ind}}(F,G_{u\to v})\ge
P_{F,G_{u\to v}}(\mathbf{x}_{u,v})$, equality must hold throughout, and
$\mathbf{x}_{u,v}$ is optimal for $G_{u\to v}$. The same argument with $u$
and $v$ interchanged gives the assertion for $G_{v\to u}$.
\end{proof}
Next, we complete the proof of Theorem~\ref{thm:cmp-ext}.
\medskip
\begin{proof}[Proof of Theorem~\ref{thm:cmp-ext}]
We begin with part~\ref{thmi} of Theorem~\ref{thm:cmp-ext}.
Among all pairs $(G,\mathbf{x})$ with nonnegative $\mathbf{x}\in \operatorname{OPT}_{\alpha,F}(G)$, choose one for which $\mathcal{T}(G,\mathbf{x})$ is maximal.
We show that $G[\supp(\mathbf{x})]$ is complete multipartite. Suppose
otherwise. We use the standard characterization that a graph is complete
multipartite if and only if every pair of nonadjacent vertices are twins.
Thus there exist nonadjacent vertices $u,v\in \supp(\mathbf{x})$ that are not
twins.
Let the equivalence classes of $u$ and $v$ in $G[\supp(\mathbf{x})]$ have sizes $a$ and $b$, respectively. 
We apply Zykov symmetrization to $u$ and $v$.
By Corollary~\ref{core}, we have 
\begin{align*}
\lambda_\alpha^{\mathrm{ind}}(F,G_{u\to v})=\lambda_\alpha^{\mathrm{ind}}(F,G_{v\to u})=\operatorname{spec}_\alpha^{\mathrm{ind}}(F,n).
\end{align*}
Consider the value
$\mathcal{T}(G_{u\to v},\mathbf{x}_{u,v})-\mathcal{T}(G,\mathbf{x})$.
Since $\mathbf{x}_{u,v}$ has the same support as $\mathbf{x}$, old
twin-pairs not involving $u$ are preserved. Indeed, symmetrization only
changes adjacencies incident with $u$; if $p,q\ne u$ were twins before the
symmetrization, then they had the same adjacency to $v$, and hence after the
symmetrization they also have the same adjacency to $u$. The vertex $u$ may
leave its former twin-class of size $a$ and join the twin-class of $v$ of size
$b$, and additional twin-pairs may also be created. Therefore,
\begin{align*}
\mathcal{T}(G_{u\to v},\mathbf{x}_{u,v})-\mathcal{T}(G,\mathbf{x})
&\ge
\tbinom{a-1}{2}+\tbinom{b+1}{2}-\tbinom{a}{2}-\tbinom{b}{2}
=b-a+1.
\end{align*}
Similarly,
\begin{align*}
\mathcal{T}(G_{v\to u},\mathbf{x}_{u,v})-\mathcal{T}(G,\mathbf{x})
\ge a-b+1.
\end{align*}
At least one of the two integers $b-a+1$ and $a-b+1$ is positive. Without
loss of generality, assume that $b-a+1>0$. Then $G_{u\to v}$ is another
extremal graph with
$\mathcal{T}(G_{u\to v},\mathbf{x}_{u,v})>\mathcal{T}(G,\mathbf{x})$,
contradicting the choice of $(G,\mathbf{x})$. Therefore
$G[\supp(\mathbf{x})]$ must be complete multipartite.

We now construct a complete multipartite extremal graph from $G$. If
$|\supp(\mathbf{x})|=n$, then $G$ itself is the desired extremal graph.
So we may assume that $|\supp(\mathbf{x})|<n$. Let
$V_1,\ldots,V_s$ be the complete multipartition of
$G[\supp(\mathbf{x})]$. Define $G^\ast$ to be the complete $s$-partite graph
obtained by putting all vertices outside $\supp(\mathbf{x})$ into $V_1$.
Then $P_{F,G^\ast}(\mathbf{x})=P_{F,G}(\mathbf{x})$, since every $\ell$-set
whose induced status changes contains at least one zero-weight vertex. Hence
$\lambda_\alpha^{\mathrm{ind}}(F,G^\ast)\ge
\operatorname{spec}_\alpha^{\mathrm{ind}}(F,n)$, while the reverse inequality
follows from the definition of $\operatorname{spec}_\alpha^{\mathrm{ind}}(F,n)$.
Thus
$\lambda_\alpha^{\mathrm{ind}}(F,G^\ast)
=\operatorname{spec}_\alpha^{\mathrm{ind}}(F,n)$. This completes the proof of
part~\ref{thmi}. The statement in part~\ref{thmii} follows directly from
Lemma~\ref{lem:const} applied to any complete multipartite graph satisfying
$\lambda_\alpha^{\mathrm{ind}}(F,G)=\operatorname{spec}_\alpha^{\mathrm{ind}}(F,n)$.

It remains to prove (iii). By part~\ref{thmi}, for every $n$ there exists
an $n$-vertex complete multipartite graph $H_n$ such that
$\operatorname{spec}_\alpha^{\mathrm{ind}}(F,n)
=\lambda_\alpha^{\mathrm{ind}}(F,H_n)$. Lemma~\ref{lem:upper-complete-multipartite}
gives $\operatorname{spec}_\alpha^{\mathrm{ind}}(F,n)
\le i(F)n^{\ell(1-1/\alpha)}$. 
For the lower bound, choose an $n$-vertex graph $G_n$ with
$I(F,G_n)=I(F,n)$. Since $I(F,n)/\binom{n}{\ell}\to i(F)$, we have $I(F,G_n)=\left(i(F)+o(1)\right)\binom n\ell$, 
and hence $\ell!I(F,G_n)=\left(i(F)+o(1)\right)n^\ell$.
Taking the uniform feasible vector on $V(G_n)$, we obtain
\begin{align*}
\operatorname{spec}_\alpha^{\mathrm{ind}}(F,n)
\ge \lambda_\alpha^{\mathrm{ind}}(F,G_n)
\ge \ell!I(F,G_n)n^{-\ell/\alpha}
=\left(i(F)+o(1)\right)n^{\ell(1-1/\alpha)}.
\end{align*}
Combining the upper and lower bounds proves (iii).
This completes the proof.
\end{proof}

\section{Proof of Theorem~\ref{thm:star}}\label{sec:star-proof}
In this section, we give a proof of Theorem~\ref{thm:star}.
The statement is trivial for $n \le t$, since no graph on $n$ vertices contains an induced copy of $K_{1,t}$. Thus, assume that $n \ge t+1$. 
By Theorem~\ref{thm:cmp-ext}, there exists an $n$-vertex complete $q$-partite graph $G = K_{n_1,\dots,n_q}$ with parts $V_1, V_2, \dots, V_q$, where
\begin{align*}
n_i&\coloneqq |V_i| \quad \text{for } i\in[q],
\end{align*}
such that $\lambda_\alpha^{\mathrm{ind}}(K_{1,t},G) = \operatorname{spec}_\alpha^{\operatorname{ind}}(K_{1,t},n)$.
Moreover, we can choose an optimal nonnegative vector
$\mathbf{x}\in \operatorname{OPT}_{\alpha,K_{1,t}}(G)$ that is constant on
each $V_i$, say with value $w_i$. Then
\begin{align}
P_{K_{1,t},G}(\mathbf{x})
&= (t+1)! \sum_{(i,j)\in [q]_2}\binom{n_i}{t} n_j w_i^t w_j.
\label{eq:star-formula}
\end{align}
where $\sum_{i\in[q]} n_i w_i^\alpha = 1$.

\begin{lemma}\label{lem:star-reduction}
For every integer $q \ge 3$, if at least three of the numbers $w_1,\ldots,w_q$ are positive, then there exists a complete $(q-1)$-partite graph $G'$ and a feasible vector $\mathbf{x}'$ such that
\begin{align*}
P_{K_{1,t},G'}(\mathbf{x}') \ge P_{K_{1,t},G}(\mathbf{x}).
\end{align*}
\end{lemma}

\begin{proof}
Since at least three of the values $w_1,\ldots,w_q$ are positive, we may
relabel the parts of $G$ such that
\begin{align*}
0 < n_1w_1 \le n_2w_2 \le n_iw_i,
\quad \text{for every } i \ge 3 \text{ with } w_i > 0.
\end{align*}
Let $G'$ be the graph obtained from $G$ by merging the two parts $V_1$ and $V_2$ into a single part. Then
$G'=K_{n_1+n_2,n_3,\ldots,n_q}$.
Define
\begin{align*}
w&\coloneqq\left(\frac{n_1w_1^\alpha+n_2w_2^\alpha}{n_1+n_2}\right)^{1/\alpha},
\end{align*}
and let $\mathbf{x}'$ be the vector assigning the value $w$ to the part
$V_1 \cup V_2$ and the value $w_i$ to $V_i$ for $i \ge 3$. Then
$\mathbf{x}'$ is feasible.

It remains to show that
$P_{K_{1,t},G'}(\mathbf{x}') - P_{K_{1,t},G}(\mathbf{x})\ge 0$. Using (\ref{eq:star-formula}), we obtain

\begin{align*}
\frac{P_{K_{1,t},G'}(\mathbf{x}') - P_{K_{1,t},G}(\mathbf{x})}{(t+1)!}
&= \biggl( \tbinom{n_1+n_2}{t}w^t - \sum_{i\in[2]} \tbinom{n_i}{t}w_i^t \biggr) \sum_{j\in[3,q]} n_j w_j \\[2mm]
& \quad + \bigl( (n_1+n_2)w - n_1w_1 - n_2w_2 \bigr) \sum_{j\in[3,q]} \tbinom{n_j}{t}w_j^t - \tbinom{n_1}{t}w_1^t n_2 w_2 - \tbinom{n_2}{t}w_2^t n_1 w_1.
\end{align*}
By the definition of $w$ when $\alpha=1$, and by H\"older's inequality when
$\alpha>1$, we have
\begin{align}
n_1w_1+n_2w_2
&\leq \left(n_1w_1^\alpha+n_2w_2^\alpha\right)^{1/\alpha}(n_1+n_2)^{(\alpha-1)/\alpha}
=(n_1+n_2)w. \label{eq:star-holder}
\end{align}
This implies 
\begin{align*}
\bigl( (n_1+n_2)w - n_1w_1 - n_2w_2 \bigr) \sum_{j\in[3,q]} \binom{n_j}{t}w_j^t \ge 0.
\end{align*}
Therefore it is sufficient to show that
\begin{align*}
\Delta(P_2)
&\coloneqq \biggl( \tbinom{n_1+n_2}{t}w^t
- \sum_{i\in[2]} \tbinom{n_i}{t}w_i^t \biggr)
\sum_{j\in[3,q]} n_j w_j 
 - \tbinom{n_1}{t}w_1^t n_2 w_2
- \tbinom{n_2}{t}w_2^t n_1 w_1
\ge 0.
\end{align*}
Define
\begin{align*}
\gamma_t(m)&\coloneqq\binom{m}{t}/m^t.
\end{align*}
The function $\gamma_t(m)$ is non-decreasing for positive integers $m$: it is zero for $m<t$, while for $m\ge t$,
\begin{align*}
\gamma_t(m)&=\frac1{t!}\prod_{j\in\{0,\ldots,t-1\}}\left(1-\frac jm\right),
\end{align*}
and each factor is non-decreasing in $m$. It follows that
$\binom{n_1+n_2}{t}w^t = \gamma_t(n_1+n_2)((n_1+n_2)w)^t$, and
$\binom{n_i}{t}w_i^t \le \gamma_t(n_1+n_2)(n_iw_i)^t$ for $i=1,2$. Thus,
\begin{align*}
\Delta(P_2)
&\ge \gamma_t(n_1+n_2) \biggl( \Big( \sum_{j\in[3,q]} n_j w_j \Big) \left( ((n_1+n_2)w)^t - (n_1 w_1)^t - (n_2 w_2)^t \right) 
 - (n_1 w_1)^t n_2 w_2 - (n_2 w_2)^t n_1 w_1 \biggr).
\end{align*}
Set
\begin{align*}
a_1&\coloneqq n_1w_1,\qquad
a_2\coloneqq n_2w_2,\qquad\text{and}\quad
S\coloneqq\sum_{j\in[3,q]}n_jw_j.
\end{align*}
Then $0<a_1\le a_2\le S$. Moreover, (\ref{eq:star-holder}) gives
$(n_1+n_2)w\ge a_1+a_2$, and
\[
(a_1+a_2)^t-a_1^t-a_2^t
=\sum_{k=1}^{t-1}\binom tk a_1^ka_2^{t-k}
\ge t a_1a_2^{t-1}.
\]
Also, $a_1^ta_2+a_2^ta_1\le 2a_1a_2^t$. 
Consequently,
\begin{align*}
\Delta(P_2)
&\ge \gamma_t(n_1+n_2)
\left(S\big((a_1+a_2)^t-a_1^t-a_2^t\big)-a_1^ta_2-a_2^ta_1\right)\\[2mm]
&\ge \gamma_t(n_1+n_2)
\left(tSa_1a_2^{t-1}-a_1^ta_2-a_2^ta_1\right)\\[2mm]
&\ge \gamma_t(n_1+n_2)
\left(ta_1a_2^t-a_1^ta_2-a_2^ta_1\right)
\ge \gamma_t(n_1+n_2)(t-2)a_1a_2^t\ge0,
\end{align*}
where the third inequality uses $S\ge a_2$.
This completes the proof.
\end{proof}
\medskip
Next, we complete the proof of Theorem~\ref{thm:star}.
\medskip
\begin{proof}[Proof of Theorem~\ref{thm:star}]
Write $G_q$ for $G=K_{n_1,\ldots,n_q}$. Let $\mathbf{x}$ be an optimal vector
of $G_q$ whose entries are identical on each part, and write $w_i$ for the
value of $\mathbf{x}$ on $V_i$. If at least three of the numbers
$w_1,\ldots,w_q$ are positive, then Lemma~\ref{lem:star-reduction} gives a
complete $(q-1)$-partite graph $G_{q-1}$ and a feasible vector
$\mathbf{x}'$ such that
\[
P_{K_{1,t},G_{q-1}}(\mathbf{x}')
\ge P_{K_{1,t},G_q}(\mathbf{x})
=\operatorname{spec}_\alpha^{\mathrm{ind}}(K_{1,t},n).
\]
Since
$\lambda_\alpha^{\mathrm{ind}}(K_{1,t},G_{q-1})\le
\operatorname{spec}_\alpha^{\mathrm{ind}}(K_{1,t},n)$ and
$\lambda_\alpha^{\mathrm{ind}}(K_{1,t},G_{q-1})\ge
P_{K_{1,t},G_{q-1}}(\mathbf{x}')$, equality holds throughout. In particular,
$\lambda_\alpha^{\mathrm{ind}}(K_{1,t},G_{q-1})
=\operatorname{spec}_\alpha^{\mathrm{ind}}(K_{1,t},n)$.
By Theorem~\ref{thm:cmp-ext}\ref{thmii}, we may again choose an optimal
vector for $G_{q-1}$ whose entries are identical on each part.

Repeating this process, we obtain an extremal complete multipartite graph
$G_{q'}$ for some $1\le q'\le q$, together with an optimal vector that is
identical on each part and has positive values on at most two parts.
Since $n\ge t+1$, we have
$\operatorname{spec}_\alpha^{\mathrm{ind}}(K_{1,t},n)>0$. Thus the optimal
vector must have at least two positive parts. Consequently, the optimal vector
on $G_{q'}$ has exactly two positive parts; denote them by $V_1$ and $V_2$.

Now merge all vertices in $V(G_{q'})\setminus (V_1\cup V_2)$ into $V_1$. The
current optimal vector assigns value $0$ to all vertices outside
$V_1\cup V_2$, so this operation does not change its value under
$P_{K_{1,t},\cdot}$. Thus the new graph has spectral value at least
$\operatorname{spec}_\alpha^{\mathrm{ind}}(K_{1,t},n)$, and the reverse
inequality follows from the definition of
$\operatorname{spec}_\alpha^{\mathrm{ind}}(K_{1,t},n)$. We finally obtain an
$n$-vertex complete bipartite graph $G^\ast$ such that
\begin{align*}
\lambda_\alpha^{\mathrm{ind}}(K_{1,t},G^\ast)
&=\operatorname{spec}_\alpha^{\mathrm{ind}}(K_{1,t},n).
\end{align*}
Since every graph appearing in the maximum on the right-hand side of
Theorem~\ref{thm:star} is an $n$-vertex graph, that maximum is at most
$\operatorname{spec}_\alpha^{\mathrm{ind}}(K_{1,t},n)$, while the graph
$G^\ast$ shows that it is at least
$\operatorname{spec}_\alpha^{\mathrm{ind}}(K_{1,t},n)$.
The result follows.
\end{proof}

\section{Proof of Theorem~\ref{thm:multipartite-balanced}}\label{sec:multipartite-proof}
We give a proof of Theorem~\ref{thm:multipartite-balanced}. Let $a$ and $r$
be integers such that $r\le 2^a-1$.
The statement is trivial for $n<ar$. Thus, we assume that $n\ge ar$. 
For convenience, we write $F$ for the complete $r$-partite graph $K_{a,\dots,a}$ with each part of size $a$.
By Theorem~\ref{thm:cmp-ext}, there exists an $n$-vertex complete $q$-partite graph $G = K_{n_1,\dots,n_q}$ with parts $V_1, V_2, \dots, V_q$, where
\begin{align*}
n_i&\coloneqq |V_i| \quad \text{for } i\in[q],
\end{align*}
such that $\lambda_\alpha^{\mathrm{ind}}(F,G)=\operatorname{spec}_\alpha^{\operatorname{ind}}(F,n)$.
Moreover, there exists an optimal nonnegative vector
$\mathbf{x}\in \operatorname{OPT}_{\alpha,F}(G)$ that is constant on each
$V_i$, say with value $w_i$. Thus
\begin{align}
P_{F,G}(\mathbf{x})
&=(ar)!\sum_{\substack{I\in\binom{[q]}r}}
\prod_{i\in I}\binom{n_i}{a}w_i^a,\notag
\end{align}
where $\sum_{i\in[q]} n_iw_i^\alpha=1$. For each $i\in[q]$, let $A_i \coloneqq\binom{n_i}{a}w_i^a$. 
Then we can rewrite $P_{F,G}(\mathbf{x})$ as follows.
\begin{align}
P_{F,G}(\mathbf{x})
&=(ar)!\sum_{\substack{I\in\binom{[q]}r}}\prod_{i\in I}A_i.
\label{eq:multipartite-formula}
\end{align}
\vspace{-2mm}
\begin{lemma}\label{lem:multipartite-size}
Under the setup above, if $w_i>0$ for some $i\in[q]$, then $n_i\ge a$ and
$A_i>0$.
\end{lemma}

\begin{proof}
It suffices to show that $n_i\ge a$. Suppose to the contrary that there exists an index $i_0\in[q]$ such that $w_{i_0}>0$ but $n_{i_0}<a$. Since $n\ge ar$, the complete $r$-partite graph $K_{a,\ldots,a,n-a(r-1)}$ contains an induced copy of $F$, so $\operatorname{spec}_\alpha^{\mathrm{ind}}(F,n)>0$. Since $\mathbf{x}$ is optimal for the extremal graph $G$, equation~\eqref{eq:multipartite-formula} gives $P_{F,G}(\mathbf{x})=\operatorname{spec}_\alpha^{\mathrm{ind}}(F,n)>0$. Hence at least one term in the elementary symmetric sum in~\eqref{eq:multipartite-formula} is positive, and therefore at least $r$ of the numbers $A_1,\ldots,A_q$ are positive.

We construct a new vector $\mathbf{x}'$ that takes the value $w'_k$ on $V_k$
for each $k\in[q]$. Choose an index $j_0 \neq i_0$ such that $A_{j_0}>0$ and
$n_{j_0}\ge a$. Set
\begin{align*}
w'_{i_0}\coloneqq 0,
\quad\text{and}\quad 
w'_k\coloneqq w_k \quad\text{for}\quad k \notin \{i_0,j_0\},
\end{align*}
and define $w'_{j_0}$ by
\begin{align*}
w'_{j_0}&\coloneqq\left(\frac{n_{j_0}w_{j_0}^\alpha+n_{i_0}w_{i_0}^\alpha}{n_{j_0}}\right)^{1/\alpha}.
\end{align*}
Then $\mathbf{x}'$ is feasible. Let
\begin{align*}
A'_k&\coloneqq\binom{n_k}{a}{w'_k}^a \quad \text{for } k\in[q].
\end{align*}
Since $w_{i_0}>0$, we have $w'_{j_0}>w_{j_0}$. Together with $A_{j_0}>0$ and $n_{j_0}\ge a$, this gives $A'_{j_0}>A_{j_0}$. For all $k\notin\{i_0,j_0\}$, we have $A'_k=A_k$. Note also that $A'_{i_0}=A_{i_0}=0$. Consequently, by (\ref{eq:multipartite-formula}), we have
\begin{align*}
P_{F,G}(\mathbf{x}')-P_{F,G}(\mathbf{x})
&=(ar)!(A'_{j_0}-A_{j_0})\sum_{\substack{J\in\binom{[q]\setminus \{j_0\}}{r-1}}}\prod_{k\in J}A_k.
\end{align*}
Since at least $r$ of the numbers $A_1,\ldots,A_q$ are positive and one of them is $A_{j_0}$, the last sum is positive. Therefore,
$P_{F,G}(\mathbf{x}')>P_{F,G}(\mathbf{x})$, which contradicts the optimality of $\mathbf{x}$. Hence, we must have $n_{i_0}\ge a$ whenever $w_{i_0}>0$.
\end{proof}
\medskip
\begin{lemma}\label{lem:multipartite-reduction}
Under the setup above, for every integer $q\ge r+1$, if at least $r+1$ of the numbers $w_1,\ldots,w_q$ are positive, then there exists a complete $(q-1)$-partite graph $G'$ and a feasible vector $\mathbf{x}'$ such that
\begin{align*}
P_{F,G'}(\mathbf{x}')\ge P_{F,G}(\mathbf{x}).
\end{align*}
\end{lemma}

\begin{proof}
Since at least $r+1$ of the values among $w_1,\ldots,w_q$ are positive, we may relabel the parts of $G$ such that
\begin{align*}
0<A_1\le A_2\le A_i
        \qquad
        \text{for every } i\ge 3 \text{ with } w_i>0.
\end{align*}
Let $G'$ be the graph obtained from $G$ by merging the two parts $V_1$ and $V_2$ into a single part. Then
$G'=K_{n_1+n_2,n_3,\ldots,n_q}$. Define
\begin{align*}
w&\coloneqq\left(\frac{n_1w_1^\alpha+n_2w_2^\alpha}{n_1+n_2}\right)^{1/\alpha}.
\end{align*}
Then
\begin{align}
n_1w_1+n_2w_2&\le (n_1+n_2)w.\label{eq:multipartite-holder}
\end{align}
Let $\mathbf{x}'$ be the vector assigning the value $w$ to the part
$V_1\cup V_2$ and the value $w_i$ to $V_i$ for $i\ge 3$. Then $\mathbf{x}'$
is feasible.
We show that $P_{F,G'}(\mathbf{x}')-P_{F,G}(\mathbf{x})\ge 0$. Writing
\begin{align*}
A'&\coloneqq\binom{n_1+n_2}{a}w^a
\end{align*}
and using (\ref{eq:multipartite-formula}), we obtain
\begin{align*}
\frac{P_{F,G'}(\mathbf{x}')-P_{F,G}(\mathbf{x})}{(ar)!}
&= (A'-A_1-A_2)
\sum_{\substack{J\subseteq[3,q]\\ |J|=r-1}}
\prod_{j\in J}A_j
 -A_1A_2
\sum_{\substack{J\subseteq[3,q]\\ |J|=r-2}}
\prod_{j\in J}A_j.
\end{align*}
For $0\le \ell\le r-1$, set
\begin{align*}
S_{\ell}&\coloneqq \sum_{\substack{J\subseteq[3,q]\\ |J|=\ell}} \prod_{j\in J}A_j,
\end{align*}
with the convention $S_0=1$.
Define
\begin{align*}
B(m)&\coloneqq\binom{m}{a}^{1/a}/m.
\end{align*}
For $m\ge a$, this function is increasing because
\begin{align*}
B(m)^a&=\frac1{a!}\prod_{j\in\{0,\ldots,a-1\}}\left(1-\frac jm\right),
\end{align*}
and the product is increasing in $m$. By Lemma~\ref{lem:multipartite-size}, we have $n_1,n_2\ge a$. Hence
\begin{align}
{A'}^{1/a}
&=B(n_1+n_2)\cdot (n_1+n_2)w \notag \\[2mm]
&\ge B(n_1+n_2)(n_1w_1+n_2w_2) 
\ge B(n_1)n_1w_1+B(n_2)n_2w_2 =A_1^{1/a}+A_2^{1/a},\notag
\end{align}
where the first inequality holds by (\ref{eq:multipartite-holder}), and the
second inequality follows from the fact that $B(m)$ is increasing. Since
$S_{r-1}\ge0$, the expression for
$P_{F,G'}(\mathbf{x}')-P_{F,G}(\mathbf{x})$ is nondecreasing in $A'$.
It suffices to show that $\Delta(P_r)\ge0$, where
\begin{align*}
\Delta(P_r)&\coloneqq \big((A_1^{1/a}+A_2^{1/a})^a-A_1-A_2\big)S_{r-1}-A_1A_2S_{r-2}.
\end{align*}
Since there are at least $r-1$ positive terms
among $A_3, \dots, A_q$, all of which are at least $A_2$, we have
\begin{align*}
(r-1)S_{r-1}
&=\sum_{\substack{J\subseteq[3,q]\\ |J|=r-2}}
\left(\prod_{j\in J}A_j\right)
\sum_{i\in[3,q]\setminus J}A_i
\ge A_2S_{r-2}.
\end{align*}
Here the inner sum is at least $A_2$ for every $J$ with $|J|=r-2$, because
such a set $J$ cannot contain all positive terms among $A_3,\ldots,A_q$.
Thus
\begin{align*}
S_{r-1}&\ge \frac{A_2}{r-1}S_{r-2}.
\end{align*}
It follows that
\begin{align*}
\Delta(P_r)
&\ge \left((A_1^{1/a}+A_2^{1/a})^a-A_1-A_2\right)
       \frac{A_2}{r-1}S_{r-2}-A_1A_2S_{r-2} \\[2mm]
&=\frac{A_2S_{r-2}}{r-1}
\left((A_1^{1/a}+A_2^{1/a})^a-A_2-rA_1\right).
\end{align*}
Let
\begin{align*}
\tau&\coloneqq(A_1/A_2)^{1/a}.
\end{align*}
Then $0<\tau\le 1$. Note that $r\le 2^a-1$, so
\begin{align*}
\left(A_1^{1/a}+A_2^{1/a}\right)^a-A_2-rA_1
&=A_2\big((1+\tau)^a-1-r\tau^a\big) \\[2mm]
&=A_2\Big( \sum_{k\in[a]}\tbinom{a}{k}\tau^k-r\tau^a \Big) 
\ge A_2\big((2^a-1)\tau^a-r\tau^a\big)\ge 0.
\end{align*}
Therefore, $\Delta(P_r)\ge 0$, and so $P_{F,G'}(\mathbf{x}')\ge P_{F,G}(\mathbf{x})$. This completes the proof.
\end{proof}
\medskip

Next, we complete the proof of Theorem~\ref{thm:multipartite-balanced}.
\medskip
\begin{proof}[Proof of Theorem~\ref{thm:multipartite-balanced}]
Write $G_q$ for $G=K_{n_1,\ldots,n_q}$. Let $\mathbf{x}$ be an optimal vector
of $G_q$ whose entries are identical on each part, and write $w_i$ for the
value of $\mathbf{x}$ on $V_i$. If at least $r+1$ of the numbers
$w_1,\ldots,w_q$ are positive, then Lemma~\ref{lem:multipartite-reduction}
gives a complete $(q-1)$-partite graph $G_{q-1}$ and a feasible vector
$\mathbf{x}'$ such that
\[
P_{F,G_{q-1}}(\mathbf{x}')
\ge P_{F,G_q}(\mathbf{x})
=\operatorname{spec}_\alpha^{\mathrm{ind}}(F,n).
\]
Since
$\lambda_\alpha^{\mathrm{ind}}(F,G_{q-1})\le
\operatorname{spec}_\alpha^{\mathrm{ind}}(F,n)$ and
$\lambda_\alpha^{\mathrm{ind}}(F,G_{q-1})\ge
P_{F,G_{q-1}}(\mathbf{x}')$, equality holds throughout. In particular,
$\lambda_\alpha^{\mathrm{ind}}(F,G_{q-1})
=\operatorname{spec}_\alpha^{\mathrm{ind}}(F,n)$.
By Theorem~\ref{thm:cmp-ext}\ref{thmii}, we may again choose an optimal
vector for $G_{q-1}$ whose entries are identical on each part.

Repeating this process, we obtain an extremal complete multipartite graph
$G_{q'}$ for some $1\le q'\le q$, together with an optimal vector that is
identical on each part and has positive values on at most $r$ parts.
Since $n\ge ar$, we have $\operatorname{spec}_\alpha^{\mathrm{ind}}(F,n)>0$.
Thus the optimal vector must have at least $r$ positive parts. Consequently,
the optimal vector on $G_{q'}$ has exactly $r$ positive parts; denote them by
$V_1,\ldots,V_r$.

Now merge all vertices in $V(G_{q'})\setminus (V_1\cup\cdots\cup V_r)$ into
$V_1$. The current optimal vector is zero on these vertices, so the value of
$P_{F,\cdot}$ at this vector is unchanged. Thus the new graph has spectral
value at least $\operatorname{spec}_\alpha^{\mathrm{ind}}(F,n)$, and the
reverse inequality follows from the definition of
$\operatorname{spec}_\alpha^{\mathrm{ind}}(F,n)$. We finally obtain an
$n$-vertex complete $r$-partite graph $G^\ast$ such that
$\lambda_\alpha^{\mathrm{ind}}(F,G^\ast)
=\operatorname{spec}_\alpha^{\mathrm{ind}}(F,n)$.
Since every graph appearing in the maximum on the right-hand side of
Theorem~\ref{thm:multipartite-balanced} is an $n$-vertex graph, that maximum
is at most $\operatorname{spec}_\alpha^{\mathrm{ind}}(F,n)$, while the graph
$G^\ast$ shows that it is at least
$\operatorname{spec}_\alpha^{\mathrm{ind}}(F,n)$.
The result follows.
\end{proof}

\section*{Acknowledgments}
Research of Liying Kang was supported by the National Natural Science Foundation of China (grant numbers 12331012, 12571375).
Research of Xizhi Liu was supported by the Excellent Young Talents Program (Overseas) of the National Natural Science Foundation of China.

\section*{Declaration on the use of AI}
The authors used generative AI tools to assist in discussing proof strategies, checking proofs, and improving exposition. 

\bibliographystyle{abbrv}
\bibliography{references}
\end{document}